\documentclass[a4paper,11pt]{article}

\usepackage{amsmath,amsthm,amssymb,fullpage,authblk,tikz}

\newtheorem{thm}{Theorem}
\newtheorem{lem}[thm]{Lemma}

\theoremstyle{definition}
\newtheorem{clm}{Claim}[thm]
\newenvironment{poc}{\begin{proof}[Proof of claim]}{\end{proof}}

\newcommand*{\abs}[1]{\lvert#1\rvert}
\newcommand{\res}{\operatorname{res}}
\newcommand{\intr}{\operatorname{int}}
\newcommand{\ext}{\operatorname{ext}}
\newcommand{\ie}{i.e.\ }
\newcommand{\nxy}{\mathcal N_{X,Y}}
\newcommand{\iij}{\mathbb{I}_{i_X<j_X}}

\title{The path minimises the average size of a connected induced subgraph}
\author{John Haslegrave}
\affil{Mathematics Institute, University of Warwick, UK}
\date{}

\begin{document}
\maketitle
\begin{abstract}We prove that among connected graphs of order $n$, the path uniquely minimises the average order of its connected induced subgraphs. This confirms a conjecture of Kroeker, Mol and Oellermann, and generalises a classical result of Jamison for trees, as well as giving a new, shorter proof of the latter.

A different proof of the main result was given independently and almost simultaneously by Andrew Vince; the two preprints were submitted one day apart.

\end{abstract}
\section{Introduction}
Connectedness is perhaps the most fundamental property of a network, and if nodes of a network may fail then the robust parts of the network, which remain connected even if all other nodes fail, are of particular interest. In particular, we might ask what a typical such part looks like. 

Such questions have a long history when the network is a tree, that is, a minimally connected graph. Jamison \cite{Jam83} studied the average order of a connected subgraph of a fixed tree of given order, that is, the average order of a subtree, showing that this invariant is minimised for the path, where it is just over one third of the total number of vertices, but at the other extreme the average proportion of vertices in a subtree can be arbitrarily close to $1$. Meir and Moon \cite{MM83} gave asymptotic results on the average value over all trees of a given order.

Subsequent work has considered special classes of trees, such as the series-reduced trees, that is, the trees with no vertex of degree $2$. Series-reduced trees arise naturally by taking the smallest element from each class of topologically equivalent trees. Jamison \cite{Jam83} conjectured that for these trees the average order of a subtree is at least half that of the original tree. This was confirmed by Vince and Wang \cite{VW10}, who also gave an upper bound of three quarters. The present author \cite{Has14} classified the sequences of trees which approach either bound.

There are two plausible ways to generalise these extremal questions to the case where the underlying graph, $G$, is not necessarily a tree. (We will always assume $G$ to be connected.) One is to ask about the average order of subtrees, that is, subgraphs which are trees; see e.g.~\cite{CGMV}. The other, perhaps more natural, is to ask about the average order of a connected induced subgraph, as alluded to above. We say a nonempty set of vertices of $G$ is a \textit{connected set} if it is the vertex set of a connected induced subgraph (here we do not consider the empty set to be a connected set). The study of the average size of a connected induced subgraph of a graph was initiated by Kroeker, Mol and Oellermann \cite{KMO18}, and further developed by Vince \cite{Vin21}. We remark that the profile of the connected induced subgraphs of $G$ will typically be very different to that of its subtrees, since every connected induced subgraph corresponds to at least one subtree on the same set of vertices, but typically larger subgraphs correspond to a greater number, biassing the average order of a subtree upwards.

Kroeker, Mol and Oellermann \cite{KMO18} conjectured that the average order of a connected induced subgraph for a graph $G$ is minimised, among connected graphs of given order, when $G$ is a path. In the case that $G$ is a tree, the average order of a connected induced subgraph is precisely the average order of a subtree, and so Jamison's result shows that the path is minimal among trees. Kroeker, Mol and Oellermann determined the minimal graph among cographs of order $n$ (which is not the path for $n\geq 4$, since it is not a cograph); subsequently, together with Balodis, they showed that the path is minimal among block graphs \cite{BMOK}.

In what follows, we write $N(G)$ for the number of connected sets of $G$, and $A(G)$ for their average order. We also use a local analogue: $N(G;v)$ denotes the number of connected sets of $G$ which contain $v$, and $A(G;v)$ denotes their average order. To avoid confusion with this notation, we use $\Gamma(v)$ for the neighbourhood of a vertex. 

Our main result confirms the above conjecture. It also gives a new, self-contained, shorter proof of Jamison's classical result for trees \cite{Jam83}.
\begin{thm}\label{main}Let $G$ be a connected graph of order $n$. Then $A(G)\geq(n+2)/3$, with equality if and only if $G$ is a path.\end{thm}
While this proof was being written up, a different proof of the main result was independently obtained by Vince \cite{Vnew}.
\section{Proof}
The proof requires two ingredients. The first, and simpler, is a bound on the local average size.
\begin{lem}\label{local}For any connected graph $G$ and any vertex $v$, $A(G;v)\geq(\abs{G}+1)/2$.\end{lem}
\begin{proof}Let $H$ be a (not necessarily connected) graph, and let $S$ be a set of vertices including at least one vertex from every component of $H$. Define a subset $U$ of vertices to be \textit{$(S,H)$-connected} if either $U=\varnothing$ or every component of $H[U]$ contains at least one vertex in $S$.
\begin{clm}The average size of an $(S,H)$-connected set is at least $\abs{H}/2$.
\end{clm}
\begin{poc}We proceed by induction on $\abs{H}$; the case $\abs{H}=1$ is trivial, so assume $\abs{H}>1$. First suppose $S$ consists of a single vertex, $x$. Set $S'=\Gamma(x)$ and $H'=H-x$; note that $S'$ meets every component of $H'$. Thus the average size of an $(S',H')$-connected set is at least $(\abs{H}-1)/2$ by the induction hypothesis. The $(S,H)$-connected sets containing $x$ are precisely the sets obtained by adding $x$ to the $(S',H')$-connected sets, and so these have average size at least $(\abs{H}+1)/2$. There is only one $(S,H)$-connected set not containing $x$ (namely, $\varnothing$), and so the average size of an $(S,H)$-connected set is at least $\frac{k}{k+1}\cdot\frac{\abs{H}+1}{2}$, where $k$ is the number of $(S,H)$-connected sets containing $x$. Every set consisting of a shortest path in $H$ from $x$ to any vertex (including the single-vertex path from $x$ to itself) is $(S,H)$-connected. Consequently $k\geq\abs{H}$, giving $\frac{k}{k+1}\cdot\frac{\abs{H}+1}{2}\geq\frac{\abs{H}}{2}$, as required.

In the case $\abs{S}>1$, we proceed similarly. Fix $x\in S$ and set $S'=\Gamma(x)\cup S\setminus\{x\}$ and $H'=H-x$. As before, the $(S,H)$-connected sets containing $x$ are precisely those obtained by adding $x$ to an $(S',H')$-connected set, and so these have average size at least $(\abs{H}+1)/2$. Set $S^*=S\setminus\{x\}$, and let $H^*$ be the induced subgraph of $H'$ consisting of those components which meet $S^*$. Write $W=V(H)\setminus V(H^*)$. By the induction hypothesis, the average size of an $(S^*,H^*)$-connected set, or equivalently of an $(S,H)$-connected set not containing $x$, is at least $\abs{H^*}/2=(\abs{H}-\abs{W})/2$. Observe that for every $(S,H)$-connected set not containing $x$ there are at least $\abs{W}$ $(S,H)$-connected sets containing $x$, obtained by adding the vertices of a shortest path from $x$ to any vertex in $W$, and all these sets are distinct. Thus, writing $\mathcal{C}_x$ for the collection of $(S,H)$-connected sets containing $x$ and $\mathcal{C}'$ for the collection of those not containing $x$, we have
\begin{align*}
\frac{1}{\abs{\mathcal{C}_x}+\abs{\mathcal{C}'}}\sum_{C\in\mathcal C_x\cup\mathcal C'}\abs{C}&\geq\frac{1}{\abs{\mathcal{C}_x}+\abs{\mathcal{C}'}}\biggl(\abs{\mathcal C_x}\frac{\abs H+1}{2}+\abs{\mathcal C'}\frac{\abs H-\abs W}{2}\biggr)\\
&=\frac{1}{1+\abs{\mathcal C_x}/\abs{\mathcal C'}}\biggl(\frac{\abs{\mathcal C_x}}{\abs{\mathcal C'}}\cdot\frac{\abs H+1}{2}+\frac{\abs{H}-\abs{W}}{2}\biggr).
\end{align*}
Since this is increasing in $\abs{\mathcal C_x}/\abs{\mathcal C'}$, and $\abs{\mathcal C_x}/\abs{\mathcal C'}\geq \abs{W}$, the average size of an $(S,H)$-connected set is at least \[\frac{\abs{W}(\abs{H}+1)/2+(\abs{H}-\abs{W})/2}{1+\abs{W}}=\frac{\abs{H}}{2}.\qedhere\]
\end{poc}
By the claim, the average size of a $(\Gamma(v),G-v)$-connected set is at least $(\abs{G}-1)/2$. Since the connected sets containing $v$ are precisely these sets with $v$ added, they have average size at least $(\abs{G}+1)/2$.
\end{proof}
This is tight when $G$ consists of a spider centred at $v$ (that is, a tree in which every vertex except $v$ has degree at most $2$) together with an arbitrary set of edges between neighbours of $v$.

The second ingredient, which may be of independent interest, shows that we may find a vertex which is in a reasonable proportion of connected sets, but is not a cutvertex.
\begin{lem}\label{tops}Let $G$ be any connected graph on $n\geq 3$ vertices. Then $G$ contains a vertex $v$ such that $G-v$ is connected and $N(G;v)\geq \frac{2N(G)}{n+1}$, with equality if and only if $G$ is a path.\end{lem}
\begin{proof}
Write $\ell=\operatorname{diam}(G)$. It is easy to verify that if $G$ is complete then any vertex $v$ will do, and that if $G$ is a path then either endvertex will do, so we may assume $2\leq\ell\leq n-2$.
Fix two vertices $v_0,v_\ell$ at distance $\ell$, and a shortest path $P$ between them. Write $v_1\cdots v_{\ell-1}$ for the internal vertices of the path. We will think of $P$ as running from left to right, with smaller indices further left. We will prove that either $N(G;v_0)\geq \frac{2N(G)}{n+1}$ or $N(G;v_\ell)\geq \frac{2N(G)}{n+1}$. Since $d(v_0,v_\ell)$ is maximal, neither vertex can be a cutvertex, so this will prove the lemma.

We define a coloured directed multigraph $H$ on the connected sets of $G$, as follows.

Let $S$ be a connected set of $G$. If $d(S,P)\geq 2$ then choose a vertex $x$ with $d(x,S)=1$ and $d(x,P)=d(S,P)-1$. Note that $S\cup\{x\}$ is also connected, since $S$ is connected and $x$ is adjacent to some vertex of $S$. Add two directed edges, one red and one blue, from $S$ to $S\cup\{x\}$. We stress that though any vertex $x$ satisfying the conditions may be chosen, the same vertex is used for both red and blue edges.
If $d(S,P)\leq 1$ and $v_0\not\in S$, let $i$ be minimal such that $d(v_i,S)=1$ and add a red edge from $S$ to $S\cup\{v_i\}$. If $d(S,P)\leq 1$ and $v_\ell\not\in S$, let $j$ be maximal such that $d(v_j,S)=1$ and add a blue edge from $S$ to $S\cup\{v_j\}$.

This construction ensures that in $H$, every vertex corresponding to a connected set not containing $v_0$ has exactly one red outgoing edge, to a connected set with exactly one additional element, whereas every vertex corresponding to a connected set containing $v_0$ has no red outgoing edge. Furthermore, every vertex has at most one incoming red edge: writing $S$ for the corresponding set, if $S\cap P=\varnothing$ then in order to have an incoming red edge there must be a unique $s\in S$ that is closest to $P$, and the only incoming red edge can be from $S\setminus\{s\}$; if $S\cap P\neq \varnothing$ then the only possible incoming red edge is from $S\setminus\{v_a\}$, where $a$ is minimal such that $v_a\in S$. Likewise every vertex corresponding to a set not containing $v_\ell$ has exactly one blue outgoing edge, and every set has at most one blue incoming edge. Consequently the subgraph containing only the red edges is a union of directed paths, each with exactly one vertex (the last vertex of the path) corresponding to a set containing $v_0$, and likewise for the blue subgraph and $v_\ell$. Note that we include some single-vertex paths, where there is a vertex with no incoming or outgoing edge of a particular colour. If there is a vertex not incident with edges of either colour, this counts as two single-vertex paths, one corresponding to each colour. See Figure \ref{fig:red-blue} for an example.

\begin{figure}
\centering
\begin{tikzpicture}[very thick]
\filldraw (1,0) circle (0.05) node[anchor=north] {$v_0$};
\filldraw (2.5,0) circle (0.05) node[anchor=north] {$v_1$};
\filldraw (4,0) circle (0.05) node[anchor=north] {$v_2$};
\filldraw (2.5,2) circle (0.05) node[anchor=south] {$a$};
\filldraw (1.75,1) circle (0.05) node[anchor=east] {$b$};
\filldraw (3.25,1) circle (0.05) node[anchor=west] {$c$};
\draw[thick] (1,0) -- (4,0) -- (2.5,2) -- cycle;
\draw[thick] (1.75,1) -- (2.5,0) -- (3.25,1);
\node at (2.5,-1) {$G$};

\draw[->, red, shorten >=.1cm] (10,3) to[out=-135, in=135] (10,2);
\draw[->>, blue, shorten >=.1cm] (10,3) to[out=-45, in=45] (10,2);
\draw[->, red, shorten >=.1cm] (10,2) to (9,1);
\draw[->, red, shorten >=.1cm] (11,1) to (10,0);
\draw[->, red, shorten >=.1cm] (12,0) to (11,-1);
\draw[->>, blue, shorten >=.1cm] (10,2) to (11,1);
\draw[->>, blue, shorten >=.1cm] (11,1) to (12,0);
\draw[->>, blue, shorten >=.1cm] (9,1) to (10,0);
\draw[->>, blue, shorten >=.1cm] (10,0) to (11,-1);

\filldraw (10,3) circle (0.05) node[anchor=south] {$\{a\}$};
\filldraw (10,2) circle (0.05) node[anchor=east] {$\{a,b\}$\,};
\filldraw (9,1) circle (0.05) node[anchor=east] {$\{a,b,v_0\}$};
\filldraw (11,1) circle (0.05) node[anchor=west] {\,$\{a,b,v_1\}$};
\filldraw (10,0) circle (0.05) node[anchor=east] {$\{a,b,v_0,v_1\}$\;};
\filldraw (12,0) circle (0.05) node[anchor=west] {$\{a,b,v_1,v_2\}$};
\filldraw (11,-1) circle (0.05) node[anchor=north] {$\{a,b,v_0,v_1,v_2\}$};
\end{tikzpicture}
\caption{A graph $G$ with path $P=v_0v_1v_2$ (left) and a component of the auxiliary digraph $H$ (right). For the set $\{a\}$ at distance $2$ from $P$, the vertex $b$ was chosen. Double-headed arrows indicate blue edges. The set $\{a,b,v_0\}$ is a blue top but not a red top.}\label{fig:red-blue}
\end{figure}
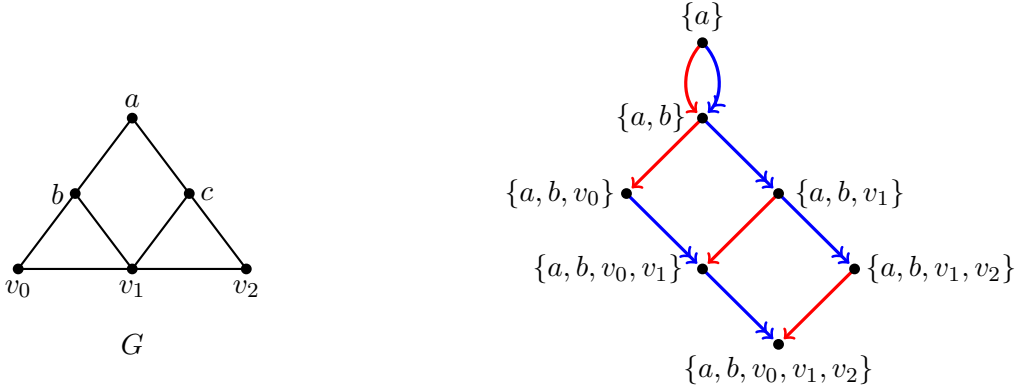

We bound the average length of all these paths (here the length of a path is the number of edges, possibly $0$). We refer to a connected set of $G$ as a ``red top'' (respectively, ``blue top'') if the corresponding vertex in $H$ has no incoming red (respectively, blue) edge. We associate each coloured path with its appropriately-coloured top. Formally, we say that a pair $(S,c)$, where $S\subseteq V(G)$ and $c\in\{\text{red},\text{blue}\}$, is a ``top'' if $S$ is a connected set of $G$ and the corresponding vertex of $H$ has no incoming edge of colour $c$. We write $\mathcal T$ for the set of all such pairs, and for $\tau=(S,c)\in\mathcal T$ we write $\ell(\tau)$ for the length of the path of colour $c$ from the vertex corresponding to $S$ in $H$. For $c\in\{\text{red},\text{blue}\}$, we write $\mathcal{T}_c$ for $\{(S,c')\in\mathcal{T}:c'=c\}$.

We make the following observations regarding which sets are tops.
\begin{clm}\label{clm:tops-meeting-P}Let $S$ be a connected set that meets $P$, and let $a$ be minimal such that $v_a\in S$. Then $S$ is a red top if and only if at least one of the following is satisfied:
\begin{itemize}
\item $S\setminus \{v_a\}$ is not connected; or
\item for some $i<a$, $v_i$ has a neighbour in $S\setminus v_a$.
\end{itemize}
An analogous statement holds for blue tops.\end{clm}
\begin{poc}Suppose $S$ is not a red top. Then $S'$ sends a red edge to $S$ for some connected set $S'$, and we must have $S'=S\setminus\{x\}$ for some $x\in S$. Since $S$ is connected and meets $P$, we have $d(P,S')\leq 1$, and so $S'$ sends a red edge to $S''=S'\cup v_{a'}$, where $a'$ is minimal such that $v_{a'}$ has a neighbour in $S'$. If $x\neq v_a$ or $v_i$ has a neighbour in $S\setminus\{v_a\}$ for some $i<a$, then $a'<a$ and $S''\neq S$, a contradiction. So the second property is satisfied and $x=v_a$, so connectedness of $S'$ gives the first property.

Conversely, if $S\setminus\{v_a\}$ is connected and $v_i$ has no neighbour in $S\setminus\{v_a\}$ for $i<a$ then it sends a red edge to $S$, so $S$ is not a red top.\end{poc}
\begin{clm}\label{clm:high-tops}Let $S$ be a connected set that does not meet $P$. Then $S$ is a red top if and only if it is a blue top.\end{clm}
\begin{poc}Suppose that $S$ is not a red top, and $S'$ sends a red edge to $S$. Since any connected set which meets $P$ or is at distance $1$ from $P$ sends a red edge to a set that meets $P$, we must have $d(S',P)\geq 2$. Since any such set $S'$ has identical red and blue outgoing edges, $S$ is not a blue top. The converse holds similarly.
\end{poc}
%As mentioned earlier, if $S$ meets $P$ then the only way that a set $S'$ can send a red edge to $S$ is if $S'=S\setminus v_a$, where $a$ is the minimal index with $v_a\in S$. This is because $S$ must be $S'$ together with a single other element, but if $v_a\in S'$ then a red edge from $S'$ must go to a set including $v_{a-1}$. Consequently $S$ is a red top whenever $S\setminus \{v_a\}$ is not connected.

%Notice also that if $S$ does not meet $P$ then any outgoing edge from $S$ goes to a set which is closer to $P$, and if $S$ meets $P$ then any outgoing edge goes to another set which meets $P$. It follows that if a set $S$ does not meet $P$, any incoming edge to $S$ comes from a set $S'$ which is at distance at least $2$ from $P$, and since any such set $S'$ has identical red and blue outgoing edges, $S$ is a red top if and only if it is a blue top.

We divide $\mathcal T$ into three parts. We say that a top $(S,c)$ is ``high'' if $S\subseteq V(G)\setminus V(P)$, ``low'' if $S\subseteq V(P)$, and ``normal'' otherwise (\ie if $S$ intersects both $V(P)$ and $V(G)\setminus V(P)$). We write $\mathcal{H}$, $\mathcal{L}$ and $\mathcal{N}$ for the sets of high, low and normal tops respectively. 

Let $x$ be any vertex in $V(G)\setminus V(P)$ that has a neighbour in $V(P)$ (since $\ell<n-1$, some such vertex exists). If $x$ has two neighbours on the path, $v_i,v_j$ with $i<j$, then by Claim \ref{clm:tops-meeting-P} $\{x,v_j\}$ is a normal red top and $\{x,v_i\}$ is a normal blue top. If there is a unique $i$ with $xv_i\in E(G)$, then by Claim \ref{clm:tops-meeting-P} $\{x,v_i,v_{i+1}\}$ is a normal red top if $i\neq \ell$ and $\{x,v_i,v_{i-1}\}$ is a normal blue top if $i\neq 0$. Thus $\abs{\mathcal N}>0$. Since any singleton set is a red and a blue top, $\abs{\mathcal{H}}\geq 2(n-\ell-1)>0$, and $\abs{\mathcal L}\geq 2(\ell+1)>0$.

If $\mathcal S$ is a nonempty subset of $\mathcal T$, we write $\mu(\mathcal{S})$ for the average length of paths corresponding to tops in $\mathcal S$, \ie $\mu(\mathcal S)=\frac{1}{\abs{\mathcal S}}\sum_{\tau\in\mathcal S}\ell(\tau)$. Notice that
\begin{equation}\mu(\mathcal T)=\frac{\abs{\mathcal T_\text{red}}\mu(\mathcal T_\text{red})+\abs{\mathcal T_\text{blue}}\mu(\mathcal T_\text{blue})}{\abs{\mathcal T_\text{red}}+\abs{\mathcal T_\text{blue}}}
=\frac{\abs{\mathcal H}\mu(\mathcal H)+\abs{\mathcal N}\mu(\mathcal N)+\abs{\mathcal L}\mu(\mathcal L)}{\abs{\mathcal H}+\abs{\mathcal N}+\abs{\mathcal L}}.\label{breakdown}\end{equation}

We first consider normal tops. This is the most complicated case, since red and blue normal tops do not necessarily coincide. We further divide the normal tops in two stages.

For a given normal top $\tau=(S,c)$, we define the ``residue'' $\res(\tau)$ to be the nonempty set $S\setminus V(P)$; note that $\res(\tau)$ need not be connected. Next, we define the ``interior'' $\intr(\tau)$ as follows. Note that $d(\res(\tau),V(P))=1$. For a set $X$ with $d(X,V(P))=1$, let $i_X$ be the minimal index $i$ such that $v_i$ has a neighbour in $X$, and let $j_X$ be the maximal such index. Now set $\intr(\tau)=S\cap\{v_k:i_{\res(\tau)}<k<j_{\res(\tau)}\}$, which may be empty (and necessarily is empty if $j_{\res(\tau)}-i_{\res(\tau)}\leq 1$). Write 
\[\nxy =\{\tau\in \mathcal{N}:\res(\tau)=X,\intr(\tau)=Y\}.\]
\begin{clm}\label{clm:NXY}Suppose that $\nxy \neq\varnothing$. Then 
$\mu(\nxy )\leq\ell/2$ if $i_X=j_X$ and $\mu(\nxy )\leq(\ell+1)/2$ otherwise.\end{clm}
\begin{poc}Note that if $(S,c)\in\nxy $ then $X\subset S\subseteq X\cup V(P)$, and $S$ is connected. In particular, if $v_a\in S$ for some $a<i_X$ then $S$ contains some shortest path from $v_a$ to $X$, and by definition of $i_X$ and $P$ every such path contains $v_{a+1},\ldots,v_{i_X}$. Thus $S\cap\{v_0,\ldots,v_{i_X}\}$ is either empty or of the form $\{v_a,\ldots,v_{i_X}\}$ for some $a\leq i_X$, and likewise for $S\cap\{v_{j_X},\ldots,v_\ell\}$.

We first prove the claim when $i_X=j_X$ (in which case necessarily $Y=\varnothing$). In this case for any $(S,c)\in\mathcal N_{X,\varnothing}$ we must have $S=X\cup\{v_a,\ldots,v_b\}$ for some $a\leq i_X\leq b$. Note that these sets are either all connected or all disconnected, and so since $\mathcal N_{X,\varnothing}\neq\varnothing$ they are all connected. If $S=X\cup\{v_{i_X}\}$ then either $(S,\text{red}),(S,\text{blue})\in \nxy $ or $(S,\text{red}),(S,\text{blue})\not\in \nxy $, depending on whether $X$ is connected. Otherwise, by Claim \ref{clm:tops-meeting-P}, $(S,\text{red})\in\nxy $ if and only if $S=X\cup\{v_{i_X},\ldots v_b\}$ for some $b>i_X$, and $(S,\text{blue})\in\nxy $ if and only if $S=X\cup\{v_a,\ldots,v_{i_X}\}$ for some $a<i_X$. Note that $\ell((X\cup\{v_{i_X},\ldots v_b\},\text{red}))=i_X$ and $\ell((X\cup\{v_a,\ldots,v_{i_X}\},\text{blue}))=\ell-i_X$.

Consequently, if $X$ is not connected we have
\[\mu(\nxy )=\frac{i_X(\ell-i_X)+(\ell-i_X)i_X}{(\ell-i_X)+i_X}\leq \frac{2(\ell^2/4)}{\ell}=\frac{\ell}{2},\]
by the AM--GM inequality, whereas if $X$ is connected we have
\[\mu(\nxy )=\frac{i_X(\ell-i_X)+(\ell-i_X)i_X+i_X+(\ell-i_X)}{(\ell-i_X)+i_X+2}\leq \frac{2(\ell^2/4)+\ell}{\ell+2}=\frac{\ell}{2}.\]
This completes the proof of the first part of the claim.

From now on we assume $i_X<j_X$, in which case $Y$ might not be empty. The argument is similar, but slightly more complicated. By our earlier remarks, any normal top $S$ with $X(S)=X$ and $Y(S)=Y$ must be of one of the following four possible forms: $X\cup Y$ (only possible if $Y\neq\varnothing$, since a normal top meets $P$); $R_a:=X\cup Y\cup \{v_a,\ldots v_{i_X}\}$ for some $a\leq i_X$; $S_b:=X\cup Y\cup \{v_{j_X},\ldots,v_b\}$ for some $b\geq j_X$; or $T_{a,b}:=X\cup Y\cup \{v_a,\ldots v_{i_X}\}\cup \{v_{j_X},\ldots,v_b\}$ with $a,b$ as before. Further, each set of the form $T_{a,b}$ is necessarily connected, since otherwise there would be no connected set which intersects $V(G)\setminus \{v_0,\ldots,v_{i_X},v_{j_X},\ldots,v_\ell\}$ in $X\cup Y$.

Suppose $S_{j_X}$ is connected. Then $S_b$ is connected for any $b\geq j_X$. By Claim \ref{clm:tops-meeting-P}, $S_b$ is a red top, since $v_{j_X}$ is the leftmost vertex in $S_b\cap V(P)$ but $v_{i_X}$ has a neighbour in $X\subseteq S\setminus\{v_{j_X}\}$. However, again by Claim \ref{clm:tops-meeting-P}, $T_{a,b}$ is not a red top, since $T_{a,b}\setminus\{v_a\}$ is connected. Thus there are $\ell-j_X+1$ red tops in $\nxy$ that contain $v_{j_X}$, and for each we have $\ell(S_b)=i_X+1$.

Alternatively, if $S_{j_X}$ is not connected then neither is $S_{b}$ for any $b\geq j_X$. By Claim \ref{clm:tops-meeting-P}, $T_{i_X,b}$ is a red top for each $b$, but $T_{a,b}$ is not a red top for any $a<i_X$. Thus there are $\ell-j_X+1$ red tops in $\nxy$ that contain $v_{j_X}$, and for each we have $\ell(T_{i_X,b})=i_X$. Similarly, there are exactly $i_X+1$ blue tops which contain $v_{i_X}$, which correspond to paths of length $\ell-j_X$ or $\ell-j_X+1$. 

Write $\nxy'$ for the remaining tops (if any) in $\nxy$, \ie red tops not containing $v_{j_X}$ and blue tops not containing $v_{i_X}$. By Claim \ref{clm:tops-meeting-P}, $R_a$ is not a red top for $a<i_X$ (either it is not connected, or it is connected but so is $R_{a+1}$). Further, $R_{i_X}$ is a red top if and only if it is connected but $X\cup Y$ is not. In particular, at most one of $R_{i_X}$ and $X\cup Y$ is a red top. Similarly the only possible blue tops in $\nxy'$ are $S_{j_X}$ and $X\cup Y$, with at most one of these being a blue top. Therefore $\abs{\nxy'}\leq 2$. Note that each potential red top $\tau\in\nxy'$ satisfies $\ell(\tau)\leq i_X+1$, and each potential blue top satisfies $\ell(\tau)\leq \ell-j_X+1$.

If $\abs{\nxy'}=0$, we have
\[\mu(\nxy)\leq\frac{(\ell-j_X+1)(i_X+1)+(i_X+1)(\ell-j_X+1)}{(\ell-j_X+1)+(i_X+1)}\leq\frac{\ell+i_X-j_X+2}{2}\leq\frac{\ell+1}{2},\]
using AM--GM and the fact that $j_X>i_X$. If $\abs{\nxy'}=2$, we likewise have
\begin{align*}\mu(\nxy)&\leq\frac{(\ell-j_X+1)(i_X+1)+(i_X+1)(\ell-j_X+1)+\ell-j_X+i_X+2}{(\ell-j_X+1)+(i_X+1)+2}\\
&\leq\frac{(\ell+i_X-j_X+2)^2/2+\ell+i_X-j_X+2}{\ell+i_X-j_X+4}\leq\frac{\ell+1}{2}.\end{align*}
Finally, if $\abs{\nxy'}=1$ then, by Claim \ref{clm:high-tops}, either $\nxy'=\{(R_{i_X},\text{red})\}$ or $\nxy'=\{(S_{j_X},\text{blue})\}$; assume without loss of generality the former. Since $\ell((R_{i_X},\text{red}))=i_X$, we have
\begin{align*}\frac{2(\ell-j_X+1)(i_X+1)+i_X}{\ell-j_X+i_X+3}&=\frac{2(\ell-j_X+3/2)(i_X+1)-1}{\ell-j_X+i_X+3}\\
&\leq\frac{(\ell-j_X+i_X+5/2)^2-2}{2(\ell-j_X+i_X+5/2)+1}\\
&=\frac{(2(\ell-j_X+i_X+5/2)+1)((\ell-j_X+i_X+5/2)/2-1/4)-7/4}{2(\ell-j_X+i_X+5/2)+1}\\
&<(\ell-j_X+i_X+5/2)/2-1/4\\
&\leq(\ell+1)/2,
\end{align*}
again using AM--GM and $i_X<j_X$. This completes the proof of the claim.\end{poc}

We next combine the high and normal tops. By Claim \ref{clm:high-tops}, if $(T,c)\in\mathcal H$ then both $(T,\text{red})$ and $(T,\text{blue})$ are in $\mathcal H$. Furthermore, the red and blue paths from $T$ both include some connected set $X$ at distance $1$ from $P$ (where possibly $X=T$), since the paths coincide at least until reaching $X$. We refer to this set as the ``extension'' $\ext((T,c))$. Note that every connected set $X$ satisfying $d(X,V(P))=1$ is the extension of exactly two high tops, since it lies on one path of each colour.

Since the distance reduces at each step, and $d(T,V(P))\leq n-\abs{V(P)}$, the length of the path from $T$ to $X$ is at most $n-\ell-2$. Now the red path proceeds through $X\cup\{v_{i_X}\}$, $X\cup\{v_{i_X},v_{i_X-1}\}$, and so on down to $X\cup\{v_{i_X},v_{i_X-1},\ldots,v_0\}$, so $\ell((T,\text{red}))\leq n-\ell-1+i_X$. Similarly, $\ell((T,\text{blue}))\leq n-\ell-1+(\ell-j_X)=n-1-j_X$. 
%The average length of both paths is therefore at most $n-1-(\ell+j_X-i_X)/2$; in particular it is at most $n-1-\ell/2$ if $i_X=j_X$, and at most $n-1-(\ell+1)/2$ otherwise.

Write $\mathcal{X}_1=\{\ext(\tau)\mid \tau\in \mathcal H\}$ and $\mathcal{X}_2=\{\res(\tau)\mid\tau\in\mathcal N\}\setminus\mathcal X_1$. For each $X\in\mathcal X_1\cup X_2$ set $\mathcal{N}_X=\{\tau\in\mathcal N:\res(\tau)=X\}$. Note that any $X\in\mathcal X_1$ is connected and thus, by Claim \ref{clm:tops-meeting-P}, if $i_X=j_X$ we have $(X\cup\{v_a,\ldots,v_{i_X}\},\text{blue})\in\mathcal N_X$ for each $a<i_X$ and $(X\cup\{v_{i_X},\ldots,v_b\},\text{red})\in\mathcal N_X$ for each $b>i_X$, whereas if $i_X<j_X$ then $(X\cup\{v_{i_X}\},\text{blue}),(X\cup\{v_{j_X}\},\text{red})\in \mathcal N_X$. Thus $\abs{\mathcal N_X}\geq 2$ for each $X\in\mathcal X_1$.
For each $X\in\mathcal X_1\cup\mathcal X_2$, we have $\mathcal{N}_X=\bigcup_Y\mathcal N_{X,Y}$, where the union is taken over all $Y$ with $\mathcal N_{X,Y}$ nonempty. Thus, by Claim \ref{clm:NXY} and averaging, $\mu(\mathcal N_X)\leq (\ell+\iij)/2$,
where $\mathbb{I}_\cdot$ is the indicator function.

For each $X\in\mathcal X_1$, write $\mathcal C_X=\mathcal N_X\cup\{\tau\in\mathcal H:\ext(\tau)=X\}$. By the remarks above, we have
\begin{align*}\mu(\mathcal C_X)&\leq\frac{\abs{\mathcal N_X}\mu(\mathcal N_X)+(n-\ell-1+i_X)+(n-1-j_X)}{\abs{\mathcal N_X}+2}\\
&\leq\frac{\abs{\mathcal N_X}\frac{\ell+\iij}2+2\Bigl(n-1-\frac{\ell+\iij}{2}\Bigr)}{\abs{\mathcal N_X}+2}\\
&=\frac{n-1}2+\frac{\Bigl(\frac{n-1-\ell-\iij}{2}\Bigr)(2-\abs{\mathcal N_X})}{\abs{\mathcal N_X}+2}\leq\frac{n-1}2,\end{align*}
since $n-\ell-2\geq 0$ and $\abs{\mathcal N_X}\geq 2$.

Additionally, for each $X\in\mathcal X_2$ we have $\mu(\mathcal N_X)\leq \frac{\ell+\iij}{2}\leq\frac{n-1}{2}$. Since we have
\[\mathcal H\cup\mathcal N=\bigcup_{X\in\mathcal X_1}\mathcal C_X\cup\bigcup_{X\in\mathcal X_2}\mathcal N_X,\]
by averaging we obtain $\mu(\mathcal H\cup\mathcal N)\leq \frac{n-1}{2}$.

The only remaining tops are the low tops. Since $P$ is a shortest path, any connected set contained in $P$ is an interval, and is a top if and only if it is a singleton, so \[\mathcal{L}=\{(\{v_i\},\text{red}),(\{v_i\},\text{blue})\mid 0\leq i\leq \ell\}.\] 
Note that $\ell((\{v_i\},\text{red}))=i$ and $\ell((\{v_i\},\text{blue}))=\ell-i$, so 
\[\mu(\mathcal L)=\frac{(\ell+1)\ell}{2(\ell+1)}=\frac{\ell}{2}<\frac{n-1}{2}.\]
Consequently, since $\abs{\mathcal L}>0$, $\mu(\mathcal H\cup \mathcal N)\leq (n-1)/2$ and $\mu(\mathcal L)<(n-1)/2$, \eqref{breakdown} gives
\begin{equation}\mu(\mathcal T)=\frac{(\abs{\mathcal H}+\abs{\mathcal N})\mu(\mathcal H\cup\mathcal N)+\abs{\mathcal L}\mu(\mathcal L)}{\abs{\mathcal H}+\abs{\mathcal N}+\abs{\mathcal L}}<\frac{n-1}{2}\label{all-tops}.
\end{equation}

We can now complete the proof. Note that $\sum_{T\in \mathcal{T}_{\text{red}}}\ell(T)=\abs{H}-\abs{\mathcal{T}_{\text{red}}}$, since the red paths form a spanning forest with $\abs{\mathcal T_{\text{red}}}$ components. Also, $\abs{H}=N(G)$, and each red path contains exactly one connected set containing $v_0$, so $\abs{\mathcal T_{\text{red}}}=N(G;v_0)$. Thus $\mu(\mathcal T_{\text{red}})={N(G)}/{N(G;v_0)}-1$, and similarly $\mu(\mathcal T_{\text{blue}})={N(G)}/{N(G;v_\ell)}-1$.
Suppose both $N(G;v_0)/N(G)\leq 2/(n+1)$ and $N(G;v_\ell)/N(G)\leq 2/(n+1)$. Then $\mu(\mathcal T_{\text{red}}),\mu(\mathcal T_{\text{blue}})\geq (n-1)/2$, and \eqref{breakdown} implies $\mu(\mathcal T)\geq (n-1)/2,$
contradicting \eqref{all-tops}. Thus we must have $N(G;v_0)/N(G)> 2/(n+1)$ or $N(G;v_\ell)/N(G)> 2/(n+1)$, as required.
\end{proof}
We are now ready to prove our main result.
\begin{proof}[Proof of Theorem \ref{main}]
We proceed by induction on $n$; the case $n=2$ is trivial. If $n\geq 3$ then we use Lemma \ref{tops} to choose a vertex $v$ with $G-v$ connected and $N(G;v)\geq \frac{2}{n+1}N(G)$, with strict inequality if $G$ is not a path. Note that $N(G)=N(G;v)+N(G-v)$, since a connected set of $G$ which does not contain $v$ is a connected set of $G-v$ and vice versa. By Lemma \ref{local}, we have $A(G;v)\geq (n+1)/2$. By the induction hypothesis, we have $A(G-v)\geq (n+1)/3$. Now
\begin{align*}A(G)&=\frac{N(G;v)A(G;v)+(N(G)-N(G;v))A(G-v)}{N(G)}\\
&\geq\frac{N(G;v)\frac{n+1}{2}+(N(G)-N(G;v))\frac{n+1}{3}}{N(G)}\\
&=\frac{n+1}{3}+\frac{N(G;v)}{N(G)}\cdot\frac{n+1}{6}\\
&\geq\frac{n+1}{3}+\frac{2}{n+1}\cdot\frac{n+1}{6}=\frac{n+2}{3},
\end{align*}
with the final inequality being strict if $G$ was not a path.
\end{proof}
\section{Final remarks}
Vince \cite{Vin21} conjectured that for graphs with minimum degree at least $3$, the average order of a connected set is at least half the order of the original graph. This may be thought of as an analogue of the result of Vince and Wang \cite{VW10} for series-reduced trees (although the latter will have some vertices of degree $1$), and, if true, would be best possible since the complete graph $K_n$ has $A(K_n)=n/2+o(1)$. It seems difficult to approach this conjecture using these methods. For the case of series-reduced trees the bound follows from the equivalent of Lemma \ref{local} together with the observation that any sufficiently large series-reduced tree $T$ has a vertex $v$ with $N(T;v)\geq\frac{n}{n+1}N(T)$. However, in the case of graphs with minimum degree $3$, or even of $3$-regular graphs, there are examples where $N(G;v)/N(G)$ is bounded away from $1$ (uniformly in $n$) for every vertex $v$. In fact, any vertex-transitive cubic graph is an example. To see this, note that $A(G)=\sum_{v\in V(G)}N(G;v)/N(G)$, and so any vertex-transitive graph $G$ satisfies $N(G;v)/N(G)=A(G)/\abs{G}$ for every vertex $v$; additionally, any cubic graph $G$ satisfies $A(G)/\abs{G}<0.95831$ by a recent result of the author \cite[Theorem 3.4]{Has21}. Thus we would require a bound on the average size of connected sets not containing $v$ which is very close to $n/2$, but $G-v$ does not have the same bound on its minimum degree. This remains an intriguing conjecture.
\section*{Acknowledgements}
This research was supported by the UK Research and Innovation Future Leaders Fellowship MR/S016325/1.

\end{document}